\documentclass[12pt]{amsart}
\usepackage{a4wide}

\newtheorem{theorem}{Theorem}[section]
\newtheorem{definition}[theorem]{Definition}
\newtheorem{proposition}[theorem]{Proposition}
\newtheorem{lemma}[theorem]{Lemma}

\newtheorem{conjecture}[theorem]{Conjecture}

\newtheorem{problem}[theorem]{Problem}

\newenvironment{rlist}
{

\begin{enumerate}}
{\end{enumerate}}

\begin{document}

\title{The quantum algebra of partial Hadamard matrices}

\author{Teo Banica}
\address{T.B.: Department of Mathematics, Cergy-Pontoise University, 95000 Cergy-Pontoise, France. {\tt teo.banica@gmail.com}}

\author{Adam Skalski}
\address{A.S.: Institute of Mathematics of the Polish Academy of Sciences, ul. \'Sniadeckich 8, 00-956 Warszawa, Poland and Faculty of Mathematics, Informatics, and Mechanics, University of Warsaw, Banacha 2, 02-097 Warszawa, Poland. {\tt a.skalski@impan.pl}}

\thanks{Both authors were partially supported by the HARMONIA NCN grant
2012/06/M/ST1/00169.}

\begin{abstract}
A partial Hadamard matrix is a matrix $H\in M_{M\times N}(\mathbb T)$ whose rows are pairwise orthogonal. We associate to each such $H$ a certain quantum semigroup $G$ of quantum partial permutations of $\{1,\ldots,M\}$ and study the correspondence $H\to G$. We discuss as well the relation between the completion problems for a given partial Hadamard matrix and completion problems for the associated submagic matrix $P\in M_M(M_N(\mathbb C))$, in both cases introducing certain criteria for the existence of the suitable completions.
\end{abstract}

\maketitle


\section*{Introduction}

A partial Hadamard matrix is a rectangular matrix with entries in the circle $\mathbb{T}$, $H\in M_{M\times N}(\mathbb T)$, whose rows are pairwise orthogonal. Observe that we have $M\leq N$. The basic examples are the $N\times N$ complex Hadamard matrices, and the $M\times N$ submatrices of these matrices.

There has been a lot of work on the partial Hadamard matrices, especially in the real case, i.e.\ $H\in M_{M\times N}(\pm1)$. See e.g. \cite{hal}, \cite{kms}, \cite{lle}, \cite{ver}.

We discuss here some classical questions, but primarily focus on quantum algebraic aspects of these matrices. In the square case, $M=N$, the idea is that any complex Hadamard matrix $H\in M_N(\mathbb T)$ produces a certain quantum subgroup $G\subset S_N^+$ of Wang's quantum permutation group \cite{wan}. This construction, inspired from \cite{jsu}, \cite{pop}, and heavily relying on Woronowicz's work in \cite{wo1}, was axiomatized in \cite{bbi}, \cite{bni}. One problem here, still open, is that of finding precise relations between the invariants of $G$ and the invariants of $H$. 

In the rectangular case, some of the results in \cite{bni} extend in a quite straightforward way, the only subtlety being the fact that Wang's quantum permutation group $S_N^+$ must be replaced with $\widetilde{S}_M^+$, the quantum semigroup of quantum partial permutations of $\{1,\ldots,M\}$. The latter, as usual in the quantum setting, is defined via its ``algebra of functions''.  Doing this kind of extension will be our main goal here. Note that this is a different type of `rectangular' generalization to that studied in \cite{bss}, where we investigated certain quantum symmetric spaces for  quantum permutation groups.

On the positive side, we will obtain indeed a correspondence $H\to G$, we will show that the correspondence behaves well with respect to the tensor products $\otimes$, and we will abstractly characterize the matrices $H$ producing semigroups $G$ which are classical. In general one can view $G$ as an algebraic invariant associated to $H$ and hope that the knowledge of $G$ should yield some information on $H$.

On the negative side, we will only partly axiomatize the algebraic structure of $\widetilde{S}_M^+$. Also, we will not consider any of the more advanced results from \cite{bbi}, \cite{bni}, notably those having something to do with Woronowicz's Tannakian duality in \cite{wo2}.

Finally, the purely rectangular setting $M<N$ brings of course a whole number of new problems. Generally speaking, the main problem regarding the partial Hadamard matrices $H\in M_{M\times N}(\mathbb T)$ is that of deciding if they complete or not into $N\times N$ complex Hadamard matrices. We will introduce and discuss here a related question, that we believe of interest: when does the associated submagic matrix $P\in M_M(M_N(\mathbb C))$ complete?

The paper is organized as follows: Section 1 is a preliminary section introducing and motivating the (algebra of functions on the) quantum semigroup of quantum partial permutations, including certain $C^*$-algebraic results, in Section 2 we develop the relation of this semigroup with the partial Hadamard matrices and combinatorial objects such as pre-Latin squares. In Section 3 we discuss completion problems for both submagic matrices and partial Hadamard matrices, and Section 4 contains a few concluding remarks. In the cases where ``rectangular'' proofs are easy modifications of the ``square'' ones available in the literature, we just sketch the arguments.

\section{Quantum partial permutations}

Let $H\in M_{M\times N}(\mathbb T)$ be a partial Hadamard matrix (i.e.\ a matrix whose rows are pairwise orthogonal), with rows denoted $R_1,\ldots,R_M\in\mathbb T^N$. Consider the vectors $\xi_{ij}=R_i/R_j\in\mathbb T^N$, where the division $R_i/R_j$ is defined coordinate-wise. We then have:
$$<\xi_{ij},\xi_{ik}>=\Big\langle\frac{R_i}{R_j},\frac{R_i}{R_k}\Big\rangle=\sum_l\frac{R_{il}}{R_{jl}}\cdot\frac{R_{kl}}{R_{il}}=\sum_l\frac{R_{kl}}{R_{jl}}=\langle R_k,R_j\rangle=\delta_{jk}$$

We have as well $<\xi_{ij},\xi_{kj}>=\delta_{ik}$, so that if $P_{ij}=Proj(\xi_{ij})$ denotes the rank one projection onto the span of the vector $xi_{ij}$ (which is the convention we adopt everywhere below), then the matrix with operator entries $(P_{ij})_{i,j=1, \ldots, M}$ is ``submagic'', in the sense that its entries are orthogonal on rows and columns.

This construction is well-known in the square case $M=N$, where it is very closely related to Wang's quantum permutations \cite{wan}. In order to deal with the general case $M\leq N$, our starting point will be the following definition.

\begin{definition}
A partial permutation of $\{1\,\ldots,N\}$ is a bijection $\sigma:X\simeq Y$, with $X,Y\subset\{1,\ldots,N\}$. We denote by $\widetilde{S}_N$ the set formed by such partial permutations.
\end{definition}

Observe that we have $S_N\subset\widetilde{S}_N$. The embedding $u:S_N\subset M_N(0,1)$ given by permutation matrices can be extended to an embedding $u:\widetilde{S}_N\subset M_N(0,1)$, as follows:
$$u_{ij}(\sigma)=
\begin{cases}
1&{\rm if}\ \sigma(j)=i\\
0&{\rm otherwise}
\end{cases}$$

Looking at the image of this embedding, we see that $\widetilde{S}_N$ is in bijection with the matrices $M\in M_N(0,1)$ having at most one 1 entry on each row and column.

\begin{proposition}
The number of partial permutations is given by
$$|\widetilde{S}_N|=\sum_{k=0}^Nk!\binom{N}{k}^2$$
that is, $1,2,7,34,209,\ldots$, and with $N\to\infty$ we have $|\widetilde{S}_N|\simeq N!\sqrt{\frac{\exp(4\sqrt{N}-1)}{4\pi\sqrt{N}}}$.
\end{proposition}

\begin{proof}
Indeed, in terms of partial bijections $\sigma:X\simeq Y$ as in Definition 1.1, we can set $k=|X|=|Y|$, and this leads to the formula in the statement. Equivalently, in the $M_N(0,1)$ picture, $k$ is the number of 1 entries, $\binom{N}{k}^2$ corresponds to the choice of the $k$ rows and $k$ columns for these $1$ entries, and $k!$ comes from positioning the 1 entries. Finally, for the asymptotic formula, see the On-Line Encyclopaedia of Integer Sequences (https://oeis.org/) sequence A002720.
\end{proof}

As an example, we have $|\widetilde{S}_2|=7$, the partial permutations of $\{1,2\}$ being the identity, the transposition, the 4 partial bijections of type $i\to j$, and the null map.

Let us discuss now the quantum version of the above notions.

\begin{definition}
Let $A$ be a unital $C^*$-algebra. A submagic matrix is a matrix $u=(u_{ij})_{i,j=1}^N\in M_N(A)$ whose entries are orthogonal projections (i.e.\ satisfy the equations $u_{ij} = u_{ij}^* =u_{ij}^2$),  pairwise orthogonal in each row and each column, so that $u_{ij} u_{ik}=0$ and $u_{ji} u_{ki}=0$ for all $i,j,k=1,\ldots, N$, $j\neq k$. We let $\widetilde{A}_s(N)$ be the universal  unital $C^*$-algebra generated by the entries of a $N\times N$ submagic matrix. The unital $^*$-algebra generated by the elements $u_{ij}$ inside $\widetilde{A}_s(N)$ will be denoted $\widetilde{\mathcal{A}}_s(N)$.
\end{definition}

Several examples of submagic matrices will appear implicitly and explicitly in what follows -- in particular they turn out to be related to partial Hadamard matrices. If in addition to the above conditions the sum of entries of a submagic matrix $u$ is equal to $1$ on each row and column, i.e.\ $\sum_{j=1}^N u_{ij} = \sum_{j=1}^N u_{ji}=1$ for all $j=1,\ldots, N$, we say that $u$ is \emph{magic}. Observe that we have a quotient map $\widetilde{A}_s(N)\to A_s(N)$, where at right we have Wang's $C^*$-algebra $A_s(N)$, generated by the entries of the universal $N\times N$ magic matrix \cite{wan}. Recall that the universality above means that $\widetilde{A}_s(N)$ is a $C^*$-algebra equipped with a certain submagic matrix $(u_{ij})\in M_N(\widetilde{A}_s(N))$ and for any other pair $B, (v_{ij})\in M_N(B)$, where $B$ is a unital $C^*$-algebra and $(v_{ij})$ is a submagic matrix, there exists a unique unital $^*$-homomorphism $\pi:A\to B$ such that $\pi(u_{ij})=v_{ij}$ for $i,j=1,\ldots,n$.

As a first remark, we have $\widetilde{A}_s(1)=\mathbb C^2$ (this is a unital $C^*$-algebra generated by a single orthogonal projection). Already at $N=2$ the situation becomes interesting.

We begin by identifying the universal (non-unital) $C^*$-algebra generated by two projections. It is well-known (see for example \cite{RaeSin}) that the universal \emph{unital} $C^*$-algebra generated by two projections, to be denoted by $B$, is isomorphic to the group $C^*$-algebra $C^*(D_{\infty})$ (where $D_{\infty}\approx \mathbb{Z}_2 \star  \mathbb{Z}_2$ is the dihedral group) and further isomorphic to the algebra $\{f\in C([0,1];M_2), f(0), f(1) \textup{ diagonal}\}$. In these identifications the canonical generating projections $P,Q$ are given respectively by $P=\frac{1}{2} (1-a)\in C^*(D_{\infty})$, $Q=\frac{1}{2} (1-b)\in C^*(D_{\infty})$  (with $a,b$ canonical free generators of $\mathbb{Z}_2 \star  \mathbb{Z}_2$) and by $P=\begin{bmatrix}1 & 0 \\ 0 & 0 \end{bmatrix} \in C([0,1];M_2)$, $Q=\begin{bmatrix} x & \sqrt{x(1-x)} \\ \sqrt{x(1-x)}  & 1-x \end{bmatrix} \in C([0,1];M_2)$. It is easy to see that there exists a character $\epsilon:C^*(D_{\infty})\to \mathbb{C}$ determined uniquely by the conditions $\epsilon(P)=\epsilon(Q)=0$ (in the $C([0,1];M_2)$ picture its existence is immediate: it is given by the formula $\epsilon(f) = f_{2,2}(1)$).

\begin{lemma}
The universal  $C^*$-algebra generated by two projections, say $D$, is isomorphic to $\textup{Ker}\, \epsilon$, which is the $C^*$-subalgebra of $B$ generated by projections $P,Q$. It is non-unital, and further isomorphic to $\{f\in C([0,1];M_2), f(0), f(1) \textup{ diagonal}, f_{2,2}(1)=0\}$.
\end{lemma}

\begin{proof}
Suppose that $p,q$ are two orthogonal projections in some $B(\mathsf{H})$. Then, by the facts stated before the lemma, there exists a unique unital representation $\pi:C^*(D_{\infty}) \to B(\mathsf{H})$ such that $\pi(P)=p$, $\pi(Q)=q$. Considering the restriction of $\pi$ to $C^*(P,Q)\subset C^*(D_{\infty})$ yields immediately the universal property of $C^*(P,Q)$. The inclusion $\textup{Ker}\, \epsilon \subset C^*(P,Q)$ is immediate. For the other one we use the fact that $\mathbb{C}[D_{\infty}]\subset C^*(D_{\infty})$ is simply the unital $^*$-algebra generated by $P$ and $Q$. It is then easy to see that  the intersection of
$\textup{Ker}\, \epsilon $ and $\mathbb{C}[D_{\infty}]$ is indeed the $^*$-algebra generated by $P$ and $Q$. Then we just take an arbitrary element $x \in \textup{Ker}\, \epsilon $ and approximate it by a sequence of elements in $\textup{Ker}\, \epsilon \cap \mathbb{C}[D_{\infty}]$ to see that $x \in C^*(P,Q)$.

The fact that $\textup{Ker}\, \epsilon$ is non-unital follows from the analysis of the concrete representation: we simply show that if $R$ is a projection in $C^*(D_{\infty})$ which dominates both $P$ and $Q$, then it must be equal to $1$ (indeed, pass to the complements and analyse possible vectors in $\ell^2(D_{\infty})$ left invariant by $R^{\perp}$). The remaining formula is immediate.
\end{proof}

Observe that $C^*(D_{\infty})$ is equipped with the coproduct, which is determined by the formulas $\Delta(P)=P\otimes P^{\perp} + P^{\perp} \otimes P$,
$\Delta(Q)=Q\otimes Q^{\perp} + Q^{\perp} \otimes Q$.

We are ready for the description of $\tilde{A}_s(2)$.

\begin{theorem}
$\tilde{A}_s(2)$ is isomorphic to the unitization of $D\oplus D$. It can be explicitly realised in the following isomorphic forms:
\begin{enumerate}
\item $ (\textup{Ker}\, \epsilon \oplus \textup{Ker}\, \epsilon) + \mathbb{C}1 \subset C^*(D_{\infty}) \oplus C^*(D_{\infty})$;
\item  $\{ (x,y) \in  C^*(D_{\infty}) \oplus C^*(D_{\infty}):\epsilon(x)=\epsilon(y)\}$;
\item $ \{f\in C([0,1];M_2\oplus M_2): f(0), f(1) \textup{ diagonal and } f(1)_{2,2}=f(1)_{4,4}\}$.
\end{enumerate}
\end{theorem}

\begin{proof}
Note first that as $D\oplus D$ is non-unital, it follows that whenever it is faithfully represented on a Hilbert space (or embedded into a unital $C^*$-algebra) then its unitization is isomorphic to the $C^*$-algebra spanned by the image and the unit in the target algebra (see \cite{Wegge}). Thus once we prove the first statement, the other two will follow (the equality of sets in (1) and (2) is easy to check).

So assume then that $\begin{bmatrix} p & q \\ r &s \end{bmatrix}$ is a sub-magic unitary, with entries in some $B(\mathsf{H})$. By the orthogonality relations we see that if $\mathsf{H}_1$ is the Hilbert space spanned by the images of $p$ and $s$ and  $\mathsf{H}_2$ is the Hilbert space spanned by the images of $q$ and $r$, then $\mathsf{H}_1 \perp \mathsf{H}_2 $ (thus we can view for example $p$ as a projection in $B(\mathsf{H}_1)$, without any further comments).  Let $P, S$ denote the canonical projections in the first copy of $D$ (say $D_1$) and $Q,R$ the ones in the second copy of $D$ (say $D_2$). By the lemma above there exist (unique) representations $\pi_1:D\to B(\mathsf{H}_1)$ and $\pi_2:D\to B(\mathsf{H}_2)$ such that $\pi_1(P)=p$, $\pi_1(S)=s$,
$\pi_2(Q)=q$, $\pi_2(R)=r$. Put $\pi=\pi_1\oplus \pi_2:D\oplus D\to B(\mathsf{H})$ and denote its extension to the unitization of $D\oplus D$ by the same letter. It is the immediate that $\pi$ is a well-defined unital representation of the unitization of $D\oplus D$ mapping $P$ to $p$ and so on (obviously it is the unique representation of the unitization of $D \oplus D$ with this property).
\end{proof}

In the general case now, $N\in\mathbb N$, observe that we have an embedding $\Psi$ and a quotient map $\Psi$, satisfying $\Psi\Phi=id$, as follows:
\begin{eqnarray*}
\Phi&:&\widetilde{A}_s(N)\subset\widetilde{A}_s(N+1)\\
\Psi&:&\widetilde{A}_s(N+1)\to\widetilde{A}_s(N)
\end{eqnarray*}

Indeed, such maps can be obtained by truncating and completing the submagic matrix, in the obvious way. In particular, $\widetilde{A}_s(N)$ is not commutative, for any $N\geq 2$. We do not know its exact structure (note that Wang's algebras $A_s(N)$ become noncommutative only for $N\geq 4$).

Let us investigate now the commutative case. We have here the following result.

\begin{proposition}
The submagic matrices with commuting entries appear as follows:
\begin{enumerate}
\item Pick a subset $S\subset\widetilde{S}_N$, and set $A=C(S)$.
\item Set $u_{ij}(\sigma)=1$ if $\sigma(j)=i$, and $u_{ij}(\sigma)=0$ otherwise.
\end{enumerate}
\end{proposition}

\begin{proof}
This follows from the Gelfand theorem. Indeed, if we are given a submagic matrix with commuting entries then we can assume its entries belong to a $C^*$-algebra of the form $C(S)$ for a certain compact set $S$, with $u_{ij}=\chi(X_{ij})$, where $X_{ij}\subset S$. Further we may assume that the algebra in question is spanned by the characteristic functions of all the intersections of type $\bigcap_{ij}X_{ij}^{e_{ij}}$ with $e_{ij}=1$ (null operation) or $e_{ij}=0$ (complementation operation), and the row/column disjointness condition on the sets $X_{ij}$ tell us precisely that the matrix $e=(e_{ij})$ must have at most 1 nonzero entry in each row and column, as desired.
\end{proof}

Finally, let us discuss some Hopf algebraic questions. Our basic result here is the following (note that here and further below for an arbitrary algebra $A$ we denote by $A_{class}$ its `commutative version', i.e.\ the quotient of $A$ by its commutator ideal).

\begin{proposition}
The following hold:
\begin{enumerate}
\item $\widetilde{S}_N$ is a semigroup, with multiplication given by $(\sigma\tau)(j)=i$ when $\tau(j)$ is defined, and $\sigma(\tau(j))=i$, and with $(\sigma\tau)(j)$ being undefined, otherwise.

\item The formula $\Delta(u_{ij})=\sum_ku_{ik}\otimes u_{kj}$ (or rather its extension given by the universal property) defines a coassociative unital $^*$-homomorphism from $\widetilde{A}_s(N)$ to $\widetilde{A}_s(N) \otimes \widetilde{A}_s(N)$. Similarly $\epsilon(u_{ij}) = \delta_{ij}$ defines a character on $\widetilde{A}_s(N)$. These maps  make $\widetilde{A}_s(N)$ a $C^*$-bialgebra, so that it can be viewed as the algebra of continuous functions on a compact quantum semigroup $\widetilde{S}_N^+$.

\item The canonical quotient map $\widetilde{A}_s(N)\to C(\widetilde{S}_N)$ commutes with the comultiplications, and thus can be interpreted as an embedding of quantum semigroups $\widetilde{S}_N\subset\widetilde{S}_N^+$.

\item We have an isomorphism $\widetilde{A}_s(N)_{class}\simeq C(\widetilde{S}_N)$, and so the semigroup $\widetilde{S}_N$ is the maximal classical subsemigroup of $\widetilde{S}_N^+$.
\end{enumerate}
\end{proposition}

\begin{proof}
The first two assertions are clear from definitions. Regarding now (3), consider the map $\Delta:C(\widetilde{S}_N)\to C(\widetilde{S}_N)\otimes C(\widetilde{S}_N)$ given by $\Delta(u_{ij})=\sum_ku_{ik}\otimes u_{kj}$. We have:
$$\Delta(u_{ij})(\sigma\otimes\tau)=\sum_ku_{ik}(\sigma)u_{kj}(\tau)=\sum_k\delta_{i,\sigma(k)}\delta_{k,\tau(j)}$$

Since at right we have products of Kronecker symbols, we conclude that:
$$\Delta(u_{ij})(\sigma\otimes\tau)=
\begin{cases}
1&{\rm if}\ \exists k,\tau(j)=k,\sigma(k)=i\\
0&{\rm otherwise}
\end{cases}$$

Thus $\Delta$ comes from the semigroup structure of $\widetilde{S}_N$, and we are done.

Finally, (4) follows from Proposition 1.6.
\end{proof}

As a first example, the isomorphism in Theorem 1.5 embeds $\widetilde{A}_s(2)$ into the Hopf algebra doubling of $C^*(D_\infty)$, with the non-standard comultiplication $\Delta(p)=p\otimes p$, $\Delta (q)=q\otimes q$. In fact, $\widetilde{A}_s(2)$ is a subalgebra of $C^*(D_\infty)_{\beta}$, where $\beta(p)=q$, $\beta (q)=p$ (for the description of the doubling construction and the notation above see \cite{sso}).

In general now, observe that $\widetilde{A}_s(N)$ admits also a ``subantipode'' map, i.e.\ a unital $^*$-antihomomorphism determined by the condition $S(u_{ij})=u_{ji}$, $i,j=1,\ldots,n$. These maps descend to $C(\widetilde{S}_N) = (\widetilde{A}_s(N))_{class}$, and correspond there, via Gelfand duality, to the unit permutation $id\in\widetilde{S}_N$, and to the usual inversion map, which to $\sigma:X\simeq Y$ associates the inverse bijection $\sigma^{-1}:Y\simeq X$. Note that the ``subantipode'' restricted to the unital $^*$-bialgebra $\widetilde{\mathcal{A}}_s(N)$ satisfies the following weakening of the usual antipode axiom:
\begin{equation}\label{subant} m_3 (S \otimes id \otimes S) \Delta_{2} = S,\end{equation}
where $m_3$ is the triple multiplication and $\Delta_2 = (\Delta \otimes id) \Delta$. Note that the standard axiom for an antipode of a Hopf algebra implies the relation \eqref{subant}.

Summarizing, Proposition 1.7, in conjunction with the formula \eqref{subant} should be regarded as just a first attempt of axiomatizing the structure of $\widetilde{S}_N^+$, which seems to be a quite special quantum semigroup.

As a conclusion to the considerations in this section, let us rephrase the last proposition diagrammatically.
We have maps as follows
$$\begin{matrix}
\widetilde{A}_s(N)&\to&A_s(N)\\
\\
\downarrow&&\downarrow\\
\\
C(\widetilde{S}_N)&\to&C(S_N)
\end{matrix}
\quad \quad \quad::\quad \quad\quad
\begin{matrix}
\widetilde{S}_N^+&\supset&S_N^+\\
\\
\cup&&\cup\\
\\
\widetilde{S}_N&\supset&S_N
\end{matrix}$$
with the bialgebras at left corresponding to the quantum semigroups at right.

Finally observe that $S_N$ is the biggest subgroup of $\widetilde{S}_N$, with the respect to the usual inversion operation $\sigma\to\sigma^{-1}$. The same holds in the quantum setting, with $S_N^+$ being the maximal quantum subgroup of $\widetilde{S}_N^+$, with respect to the subantipode map $S(u_{ij})=S(u_{ji})$.

It is not clear how exactly the general theory of quantum partial permutations can be further developed. One key problem is that of the general study of the `subantipode' axiom, \eqref{subant}. Another problem is that of understanding the precise structure of $\widetilde{A}_s(N)$ at small values of $N$, and particularly at $N=3$.

\section{The bialgebra image construction}

In this section we develop the bialgebra image construction for the partial Hadamard matrices, by following the work on the Hopf images in \cite{bbi}, \cite{bni}. Begin with the following result.

\begin{proposition}
Let $H\in M_{M\times N}(\mathbb T)$ be partial Hadamard, with rows $R_1,\ldots,R_M\in\mathbb T^N$ and let $(u_{ij})_{i,j=1}^M$ be the canonical submagic unitary matrix of generators of $\widetilde{A}_s(M)$. Then $u_{ij}\to Proj(R_i/R_j)$ defines a representation $\pi_H:\widetilde{A}_s(M)\to M_N(\mathbb C)$.
\end{proposition}

\begin{proof}
As explained at the beginning of section 1, the $M\times M$ array of vectors $\xi_{ij}=R_i/R_j\in\mathbb T^N$ has the property that each of its rows and columns consists of pairwise orthogonal vectors. But this means that the corresponding matrix of rank one projections $P_{ij}=Proj(\xi_{ij})$ is submagic in the sense of Definition 1.3, and we are done.
\end{proof}

We use now the fact that the notion of Hopf image, axiomatized in \cite{bbi}, extends in a straightforward way to the quantum semigroup setting (we call it the bialgebra image here, as we do not a Hopf algebra structure). Instead of developing again the general theory here let us just record the construction that we are interested in. Minimality in the next formulation is understood in universal terms, i.e.\ as the existence of unique canonical maps for any other factorisation (see \cite{bbi} for details).

\begin{proposition} \label{factHad}
There exists a minimal unital $^*$-algebra with comultiplication $\mathcal{A}$ producing a factorisation of type
$$\pi_H:\widetilde{\mathcal{A}}_s(M)\to A\to M_N(\mathbb C)$$
with the map on the left being a morphism of unital $^*$-algebras with comultiplications. Slightly abusing the language we write $\mathcal{A}=C(G(H))$, and call $G(H)\subset\widetilde{S}_M^+$ the quantum semigroup associated to $H$.
\end{proposition}

\begin{proof}
As already mentioned, this follows by reasoning as in \cite{bbi}. Indeed, the existence of $\mathcal{A}$, having the universal property in the statement, can be shown by dividing $\widetilde{\mathcal{A}}_s(M)$ by the sum of all the comultiplication-stable ideals contained in $\ker\pi_H$. See \cite{bbi}.
\end{proof}

Note that the construction above naturally works for any unital $^*$-homomorphism from $\widetilde{\mathcal{A}}_s(M)$ to a unital $^*$-algebra, not necessarily arising from a partial Hadamard matrix.
As a first example of the construction in Proposition \ref{factHad}, we have the following fact.

\begin{proposition}
Assume that $H\in M_{M\times N}(\mathbb T)$ is such that the projections $P_{ij}$ commute and let $\mathcal{A}\subset M_N(\mathbb C)$ be the algebra generated by these projections. Write $\mathcal{A}=C(S)$ with $S\subset\widetilde{S}_M^+$, as in Proposition 1.6. Then $G=<S>$ is the semigroup generated by $S$.
\end{proposition}

\begin{proof}
This is clear from definitions, because our representation $\pi_H:\widetilde{A}_s(M)\to C(S)$ factorizes through $C(<S>)$, and this factorization is minimal.
\end{proof}

Let us investigate now the tensor product problem. Given two partial Hadamard matrices $H\in M_{m\times n}(\mathbb T)$, $K\in M_{M\times N}(\mathbb T)$ we let $H\otimes K\in M_{mM\times nN}(\mathbb T)$ be the matrix given by $(H\otimes K)_{ia,jb}=H_{ij}K_{ab}$. This matrix is then partial Hadamard, and we have the following equality.

\begin{proposition}
$G(H\otimes K)=G(H)\times G(K)$.
\end{proposition}

\begin{proof}
This follows as in \cite{bni}, so we just sketch main arguments. Indeed, let us identify $\mathbb C^{nN}=\mathbb C^n\otimes\mathbb C^N$, by using the lexicographic order on the double indices $jb$. Since the $ia$-th row of $H\otimes K$ is the vector $(H_{ij}\otimes K_{ab})_{jb}$, the corresponding submagic basis is $\xi_{ia,jb}=\xi_{ij}\otimes\xi_{ab}$, and so the corresponding submagic matrix is $P_{ia,jb}=P_{ij}\otimes P_{ab}$. Here, and in what follows, we agree to use the common letter $P$ for the submagic matrices associated to $H,K,H\otimes K$, with the indices telling us which of these $P$ matrices we are exactly talking about.

Consider now the following diagram of algebras:
$$\begin{matrix}
\widetilde{A}_s(m)\otimes \widetilde{A}_s(M)&\to& C(G(H))\otimes C(G(K))&\to& M_n(\mathbb
C)\otimes M_N(\mathbb C)\\
\uparrow &&&&\downarrow\\
\widetilde{A}_s(mM)&\to&C(G(H\otimes K))&\to&M_{nN}(\mathbb C)
\end{matrix}$$

Here the arrow at right comes from $\mathbb C^{nN}=\mathbb C^n\otimes\mathbb C^N$, and the arrow at left is given by $u_{ia,jb}\to u_{ij}\otimes u_{ab}$. At the level of generators, we have the following diagram:
$$\begin{matrix}
u_{ij}\otimes u_{ab}&\to&w_{ij}\otimes w_{ab}&\to& P_{ij}\otimes
P_{ab}\\
\uparrow &&&&\downarrow\\
u_{ia,jb}&\to&w_{ia,jb}&\to&P_{ia,jb}
\end{matrix}$$

Observe now that this diagram commutes. We conclude that the factorization associated to $H\otimes K$ factorizes indeed through $C(G(H))\otimes C(G(K))$, and this gives the result.
\end{proof}

Consider now again the commutative situation.

\begin{definition}
A pre-Latin square is a matrix $L\in M_M(1,\ldots,N)$ having the property that its entries are distinct, on each row and each column.
\end{definition}

Observe that we must have $M\leq N$, and that the usual Latin squares are precisely the pre-Latin squares at $M=N$. Now assume that we are given such a pre-Latin square $L$, and an orthogonal basis $\xi=(\xi_1,\ldots,\xi_N)$ of $\mathbb C^N$. With this data in hand, we can construct a submagic matrix $P\in M_M(M_N(\mathbb C))$, according to the formula $P_{ij}=Proj(\xi_{L_{ij}})$.

On the other hand, given $L$ as above, to any $x\in\{1,\ldots,N\}$ we can associate the partial permutation $\sigma_x\in\widetilde{S}_M$ given by $\sigma_x(j)=i\iff L_{ij}=x$. We denote by $G\subset\widetilde{S}_M$ the semigroup generated by $\sigma_1,\ldots,\sigma_N$, and call it the semigroup associated to $L$.

With these notations, we have the following result. The interest in submagic matrices built of projections of rank-one is clearly motivated by the case of partial Hadamard matrices, see also Problem 2.7 below.

\begin{theorem}
For a submagic matrix $P\in M_M(M_N(\mathbb C))$ formed of rank one projections, the following are equivalent:
\begin{enumerate}
\item The projections $P_{ij}$ pairwise commute.

\item $P$ comes from a pre-Latin square $L$.
\end{enumerate}
In addition, if so is the case, the bialgebra image of $\pi_P:\widetilde{A}_s(M)\to M_N(\mathbb C)$ is the algebra $C(G)$, where $G\subset\widetilde{S}_M$ is the semigroup associated to $L$.
\end{theorem}

\begin{proof}
$(2)\implies(1)$ is clear. For $(1)\implies(2)$, we use the fact that two rank one projections commute when their images are equal, or orthogonal. Thus, if we denote by $X$ the set formed by the images of the projections $P_{ij}$, and pick for any $E\in X$ a nonzero vector $\xi_E\in E$, the system $\{\xi_E|E\in X\}$ is orthogonal. Now by enlarging $X$ if needed, we can assume that $|X|=N$, and that $\{\xi_E|E\in X\}$ is an orthogonal basis of $\mathbb C^N$.

Let us label now $1,\ldots,N$ the elements of $X$. Then our orthogonal basis  of $\mathbb C^N$ is now written $\xi_1,\ldots,\xi_N$, and for any $i,j\in\{1,\ldots,M\}$ we can write $P_{ij}=Proj(\xi_{L_{ij}})$, with $L_{ij}\in\{1,\ldots,N\}$. But, with this notation in hand, the submagic condition on $P$ translates into the fact that $L$ is a pre-Latin square, and this finishes the proof of $(1)\implies(2)$.

For the last assertion now, observe first that, since the entries of $P$ commute, the representation $\pi_P$ factorizes through $C(\widetilde{S}_M)$. On the other hand, we have as well an embedding $Im(\pi_P)\subset C(1,\ldots,N)$, given by $P_{ij}\to\delta_{L_{ij}}$. Thus, we are left with computing the bialgebra image of the map $\varphi:C(\widetilde{S}_M)\to C(1,\ldots,N)$ given by $\varphi(u_{ij})=\delta_{L_{ij}}$.

In order to solve this latter problem, we can use Gelfand duality. Indeed, $\varphi$ must come from a map $\sigma:\{1,\ldots,N\}\to\widetilde{S}_M$, according to the duality formula $\varphi(f)(x)=f(\sigma_x)$. But, with $f=u_{ij}$, this formula tells us that we must have $\delta_{L_{ij}}(x)=u_{ij}(\sigma_x)$, and so that $\sigma_x\in\widetilde{S}_M$ must be given by $\sigma_x(j)=i\iff L_{ij}=x$. Now since the bialgebra image that we are computing is $C(G)$, with $G=<x_1,\ldots,x_N>$, this gives the result.
\end{proof}

The above result raises the following question.

\begin{problem}
What are the partial Hadamard matrices $H\in M_{M\times N}(\mathbb T)$ having the property that the projections $P_{ij}=Proj(R_i/R_j)$ commute (equivalently, $P=(P_{ij})$ comes from a pre-Latin square)?
\end{problem}

In the square case, $M=N$, the only solutions are the generalized Fourier matrices, $F_G=F_{N_1}\otimes\ldots\otimes F_{N_k}$ with $G=\mathbb Z_{N_1}\otimes\ldots\otimes\mathbb Z_{N_k}$ finite abelian group. See \cite{bni}.

In the purely rectangular case, $M<N$, the situation is much more complicated. As an example, at $M=2$ any matrix of type $H=\begin{pmatrix}1&\ldots&1\\ a_1&\ldots&a_N\end{pmatrix}$ is a solution, provided that the numbers $a_1,\ldots,a_N\in\mathbb T$ satisfy the equations $\sum_ia_i=\sum_ia_i^2=0$.

\section{Completion problems}

Given a partial Hadamard matrix $H\in M_{M\times N}(\mathbb T)$, one interesting problem is that of deciding whether this matrix extends or not to a $N\times N$ complex Hadamard matrix.

In the real case, i.e.\ $H\in M_{M\times N}(\pm1)$, the existence of the completion is known to be automatic at any $M=N-K$, with $K\leq 7$. See \cite{hal}, \cite{ver}, and for more recent analysis of related completion problems \cite{cflw}.

In the general complex case, however, very little seems to be known about this question. Perhaps most illustrating here is the completion problem for matrices $H\in M_{4\times 5}(\mathbb T)$. Here we state the following conjecture.

\begin{conjecture}
Every partial Hadamard matrix in $H\in M_{4\times 5}(\mathbb T)$ completes to a Hadamard matrix in $M_{5\times 5} (\mathbb T)$.
\end{conjecture}

Proving the above statement  would probably require generalizing  Haagerup's result in \cite{haa}, which is likely to be a difficult task. Let us record, however, a related  linear algebra result. Let $H\in M_{(N-1)\times N}(\mathbb T)$, and denote by $R_1,\ldots,R_{N-1}\in\mathbb T^N$ its rows, and by $C_1,\ldots,C_N\in\mathbb T^{N-1}$ its columns.

\begin{proposition}
For $H\in M_{(N-1)\times N}(\mathbb T)$ partial Hadamard, the following are equivalent:
\begin{enumerate}
\item $H$ is completable to a $N\times N$ complex Hadamard matrix.

\item $|\det H^{(j)}|$ is independent from $j$, where $H^{(j)}$ is obtained from $H$ by removing $C_j$.

\item $|\det H^{(j)}|=N^{N/2-1}$ for any $i$, where $H^{(j)}$ is as above.
\end{enumerate}
If any/all of the above conditions hold, the completion is obtained by setting $H_{Nj} = (-1)^{j}N^{1-N/2} \overline{\det H^{(j)}}$.
\end{proposition}

\begin{proof}
We use the well-known fact that $span(R_1,\ldots,R_{N-1})^\perp=\{\lambda Z: \lambda \in \mathbb C\}$, where $Z\in\mathbb C^N$ is the (clearly non-zero!) vector having components $Z_j=(-1)^j\overline{\det H^{(j)}}$. Indeed, if we denote by $H_i$ the square matrix obtained from $H$ by adding a first row equal to $R_i$, then:
$$<R_i,Z>=\sum_jH_{ij}\overline{Z}_j=\sum_j(-1)^jH_{ij}\det H^{(j)}=\det H_i=0$$

This gives $(1)\iff(2)$, as the existence of a completion is equivalent to the fact that $span(R_1,\ldots,R_{N-1})^\perp$ contains a vector in $\mathbb C^N$ with all entries of the absolute value $1$. Further $(3)\Longrightarrow(2)$ is obvious. Finally we show $(2)\Longrightarrow(3)$. Write $c=|\det H^{(j)}|$ and let  $M\in M_N(\mathbb T)$ be the complex Hadamard completing $H$. The proof above shows that the last row of $M$ must be given by $c^{-1} Z$. As we have $\frac{1}{\sqrt{N}}M\in U_N$, where $U_N$ denotes the set of unitary $N \times N$ matrices, $|\det M|=N^{N/2}$, simply because the absolute value of a determinant of a unitary matrix is equal to $1$. It remains to compute this determinant by the expansion with respect to the last row, which gives
\[ \det M = \sum_{j=1}^N c^{-1} (-1)^{N+j} (-1)^j\overline{\det H^{(j)}} \det H^{(j)} = (-1)^N c N.\]
This means that $c= N^{N/2-1}$ and shows also the last statement in the proposition.
\end{proof}

For more information on minors of Hadamard matrices we refer to the articles \cite{kms}, \cite{kms2}, see also \cite{bns}, where square submatrices of Hadamard matrices are considered.

Let us investigate now the completion problem for submagic matrices. We first study the case $M=N-1$.

\begin{proposition}
For a submagic matrix $P\in M_{N-1}(A)$, the following are equivalent:
\begin{enumerate}
\item $P$ is completable to a magic matrix $\widetilde{P}\in M_N(A)$.

\item The projections $P_{iN}=1-\sum_jP_{ij}$ are pairwise orthogonal.

\item The projections $P_{Nj}=1-\sum_iP_{ij}$ are pairwise orthogonal.

\item The element $P_{NN}=(\sum_{ij}P_{ij})-N+2$ is a projection.
\end{enumerate}
\end{proposition}

\begin{proof}
Since a magic completion of $P$ must necessarily use the projections introduced in (2) and (3), we have $(1)\implies(2)$ and $(1)\implies(3)$. On the other hand, checking that some projections are pairwise orthogonal is the same as checking that their sum is a projection, and this gives $(2)\iff(4)$ and $(3)\iff(4)$, because the projection in question is $1-P_{NN}$. Finally, we have $(4)\implies(1)$, because we can complete $P$ with the projections in (2,3,4).
\end{proof}

In the general case now, we have the following observation.

\begin{proposition}
Assume that $K, M \in \mathbb N$, $P\in M_M(A)$ is submagic, and let $N=M+K$. If $P$ completes to a magic unitary $\widetilde{P}\in M_N(A)$ then $\sum_{i,j=1}^M P_{ij}\geq M-K$.
\end{proposition}

\begin{proof}
Let us write the magic unitary in the statement as follows:
$$\widetilde{P}=\begin{pmatrix}P&Q\\ R&S\end{pmatrix}$$

Now if we denote by $X\to X'$ the sum of entries operation, then $P'+Q'=M$ and $Q'+S'=K$. Thus we have $P'-S'=M-K$, and so $P'\geq M-K$.
\end{proof}

In the commutative case the situation, unsurprisingly, can be described completely.

\begin{proposition}
Let $P\in M_M(A)$ be a submagic matrix, with $A$ commutative, and write $A=C(S)$ with $S\subset\widetilde{S}_M$, as in Proposition 1.6. Then the following are equivalent:
\begin{enumerate}
\item $P$ completes to a magic matrix $\widetilde{P}\in M_N(A)$.

\item Each $\sigma\in S$ has at most $K$ undefined values, where $K=N-M$.

\item We have $\sum_{ij}P_{ij}\geq M-K$.
\end{enumerate}
\end{proposition}

\begin{proof}
$(1)\implies(3)$ follows from Proposition 3.3. Also, $(2)\iff(3)$ is clear, because $\sum_{ij}P_{ij}(\sigma)=|Im(\sigma)|$. In order to prove now $(2)\implies(1)$, consider the set $\widetilde{S}_{M,N}\subset\widetilde{S}_M$ consisting of partial permutations having at most $K$ undefined values.

Given $\sigma\in\widetilde{S}_{M,N}$, we define $\sigma'\in S_N$ in the following way. We write $\sigma:X\simeq Y$, with $X,Y\subset\{1,\ldots,M\}$, then we set $L=|X^c|=|Y^c|$ and we write $X^c=\{x_1,\ldots,x_L\}$ and  $Y^c=\{y_1,\ldots,y_L\}$ with $x_1<\ldots<x_L$ and $y_1<\ldots<y_L$, and finally we set:
$$\sigma'(i)=
\begin{cases}
\sigma(i)&\ if\ i\in X\\
M+r&\ if\ i=x_r\\
y_r&\ if\ i=M+r\\
i&\ if\ i>M+L
\end{cases}$$

This construction $\sigma\to\sigma'$ produces an embedding $\widetilde{S}_{M,N}\subset S_N$, and hence an embedding $S\subset S_N$, having the property $\sigma'(j)=i\iff\sigma(j)=i$. Now if we let $\varphi:C(S_N)\to C(S)$ be the map given by $\varphi(f)(\sigma)=f(\sigma')$, and we set $\widetilde{P}_{ij}=\varphi(u_{ij})$, we are done.
\end{proof}

The last proof shows in particular that if a magic submatrix with commuting entries can be completed, the completion can also be chosen so that the entries commute.

Returning to the noncommutative situation we can describe what happens when $N=2$.

\begin{proposition}
Any $2\times 2$ submagic matrix can be completed to a $4\times 4$ magic matrix, but not necessarily to a $3\times 3$ magic matrix.
\end{proposition}

\begin{proof}
Begin by certain general considerations regarding 2 by 2 partial magic unitaries. We want to understand the structure of the following matrix of projections: $\begin{bmatrix} p & r \\ s & q\end{bmatrix}$, with the relations as follows: $r,s \leq p^{\perp}$, $r,s \leq q^{\perp}$, and we think of $p,q,r,s$ as projections on some concrete Hilbert space $\mathsf{H}$. Then $\mathsf{H}$  splits into two pieces, and all projections in question leave them invariant (this is a consequence of the fact that $p\wedge q$ commutes with all $p,q,r,s$.

On these pieces we have respectively
\begin{rlist}
\item $p^{\perp} \wedge q^{\perp}=0$ (and then $r=s=0$);
\item $p^{\perp} \wedge q^{\perp}=1$ (and then $p=q=0$ and $r,s$ are arbitrary).
\end{rlist}

It is easy to see that we can `complete' the 2 by 2 magic unitary considering these pieces separately. In other words it suffices to see that  $\begin{bmatrix} p & 0 \\ 0 & q\end{bmatrix}$ and $\begin{bmatrix} 0 & r \\ s & 0\end{bmatrix}$ can be completed to a 4 by 4 magic unitary, and also to see that 3 by 3 completion is impossible if for example $p=q\neq 1$.
\end{proof}

We can conclude from the above the following statement, where by $\widetilde{A}_s(M,N)$ we denote the universal $C^*$-algebraic quotient of $\widetilde{A}_s(M)$ having the property that the relation $\sum_{ij}P_{ij}\geq M-K$ holds, where $N=M+K$.

\begin{theorem}
The universal quotient of $\widetilde{A}_s(M)$  having the property that a submagic matrix $P\in M_M(A)$ is completable to a magic matrix $\widetilde{P}\in M_N(A)$ if and only if the representation $\pi_P:\widetilde{A}_s(M)\to A$ factorizes through $\widetilde{A}_s(M,N)$ exists, and is equal to $\widetilde{A}_s(M,N)$, in the following cases:
\begin{enumerate}
\item $M=N$: here $\widetilde{A}_s(N,N)=\widetilde{A}_s(N)$.
\item $M=N-1$: here $\widetilde{A}_s(N-1,N)=A_s(N)$.
\item $M=2,N=4$: here $\widetilde{A}_s(2,4)=\widetilde{A}_s(2)$.
\end{enumerate}
In addition, we have $\widetilde{A}_s(M,N)_{class}=C(\widetilde{S}_{M,N})$, where $\widetilde{S}_{M,N}\subset\widetilde{S}_M$ is the set of partial permutations having at most $N-M$ undefined entries, and the latter algebra is the universal quotient of $\widetilde{A}_s(M)_{class}$  having the property that a commutative submagic matrix $P\in M_M(A)$ is completable to a commutative magic matrix $\widetilde{P}\in M_N(A)$ if and only if the representation $\pi_P:\widetilde{A}_s(M)\to A$ factorizes through $\widetilde{A}_s(M,N)_{class}$.
\end{theorem}

\begin{proof}
This follows indeed from the earlier results.
\end{proof}

Finally, let us go back to the completion problem for partial Hadamard matrices. Here is a related result, by using the notations in Proposition 3.1 above, and its proof.

\begin{proposition}
For $H\in M_{(N-1)\times N}(\mathbb T)$ partial Hadamard, the following are equivalent:
\begin{enumerate}
\item $P_{ij}=Proj(R_i/R_j)$ is completable to a $N\times N$ magic matrix.

\item With $G_{kl}=\frac{1}{N}|<C_k,C_l>|^2$, the matrix $G-(N-2)I$, where $I$ denotes the identity matrix,  is a projection.

\item With $span(R_1,\ldots,R_{N-1})^\perp=\mathbb CZ$ we have $HDH^*=c1_N$, where $D=diag|Z_i|^2$.
\end{enumerate}
\end{proposition}

\begin{proof}
$(1)\iff(2)$ Since for $\xi\in\mathbb C^N$ we have $Proj(\xi)=\frac{1}{||\xi||^2}(\xi_k\bar{\xi}_k)_{kl}$, we obtain:
$$P_{ij}=\frac{1}{N}\left(\frac{H_{ik}}{H_{jk}}\cdot\frac{H_{jl}}{H_{il}}\right)_{kl}=\frac{1}{N}\left(\frac{H_{ik}}{H_{il}}\cdot\frac{H_{jl}}{H_{jk}}\right)_{kl}$$

Now by summing over $i,j$, we obtain:
$$\sum_{ij}P_{ij}=\frac{1}{N}(<C_k,C_l><C_l,C_k>)_{kl}=\frac{1}{N}(|<C_k,C_l>|^2)_{kl}$$

The result follows now from Proposition 3.2.

$(1)\iff(3)$ The point here is that the submagic matrix $P=(P_{ij})$ completes if and only if the corresponding submagic basis $\xi=(\xi_{ij})$ completes. Now since $\xi_{ij}=R_i/R_j$, we have no choice for the completion, because we must set $\xi_{iN}=R_i\cdot\bar{Z}$ and $\xi_{Ni}=Z/R_j$:
$$\xi=\begin{pmatrix}
R_1/R_1&R_1/R_2&\ldots&R_1/R_{N-1}&R_1\cdot\bar{Z}\\
R_2/R_1&R_2/R_2&\ldots&R_2/R_{N-1}&R_2\cdot\bar{Z}\\
\ldots&\ldots&\ldots&\ldots&\ldots\\
R_{N-1}/R_1&R_{N-1}/R_2&\ldots&R_{N-1}/R_{N-1}&R_{N-1}\cdot\bar{Z}\\
Z/R_1&Z/R_2&\ldots&Z/R_{N-1}&*
\end{pmatrix}$$

We conclude that $P$ completes if and only if the vectors $\xi_{iN}=R_i\cdot\bar{Z}$ are orthogonal, or, equivalently, if the vectors $\xi_{Nj}=Z/R_j$ are orthogonal. But, we have:
$$<\xi_{iN},\xi_{jN}>=<R_i\cdot\bar{Z},R_j\cdot\bar{Z}>=\sum_kH_{ik}\bar{H}_{jk}|Z_k|^2=(HDH^*)_{ij}$$

Thus $P$ completes if and only if $HDH^*$ is diagonal. On the other hand, we have:
$$(HDH^*)_{ii}=\sum_kH_{ik}|Z_k|^2\bar{H}_{ik}=\sum_k|Z_k|^2$$

Thus $P$ completes if and only if $HDH^*$ is a multiple of identity, as claimed.
\end{proof}

A quick comparison between Proposition 3.1 and Proposition 3.7 might suggest that, by some linear algebra miracle, the conditions found there are in fact all equivalent. We do not know if it is so, and we have no further results here.

\section{Conclusion}

We have seen in this paper that the well-known construction $H\to G$, associating a quantum permutation group $G\subset S_N^+$ to a complex Hadamard matrix $H\in M_N(\mathbb T)$, has a straightforward extension to the partial Hadamard matrix case, $H\in M_{M\times N}(\mathbb T)$.

In order for our construction to be more effective, the quantum semigroup $\widetilde{S}_M^+$ of quantum partial permutations of $\{1,\ldots,M\}$ would need to be better understood; and there are several interesting algebraic questions here, raised throughout the paper.

In addition, we have the following big question: is there any way of developing some classical analysis tools for this compact quantum semigroup $\widetilde{S}_M^+$? Note that this is true for Wang's quantum permutation group $S_N^+$, cf. \cite{bco} and subsequent papers.

\vspace*{0.5 cm}

\noindent \textbf{Acknowledgment:} We would like to thank the referee for careful reading of our manuscript and several useful comments improving the presentation.

\end{document}